\documentclass[11pt, twoside]{article}
\usepackage{amsfonts,amssymb,amsmath,amsthm}
\usepackage{graphicx}
\usepackage[all]{xy}
\usepackage{multirow} % Required for multirows
\usepackage{psfrag,xmpmulti,amscd,color,pstricks, import}

 \divide\oddsidemargin by 2
\newtheorem{thm}{Theorem}[section]
\newtheorem{cor}[thm]{Corollary}
\newtheorem{lem}[thm]{Lemma}
\newtheorem{prop}[thm]{Proposition}

\theoremstyle{definition}
\newtheorem{defn}[thm]{Definition}

\newtheorem*{rem*}{Remark}

\newtheorem*{thmA}{Theorem A}

\numberwithin{equation}{section}

\definecolor{OrangeRed}{cmyk}{0,0.6,1,0}            % half magenta only, full yellow
\definecolor{DarkBlue}{cmyk}{1,1,0,0.20}
\definecolor{DarkGreen}{cmyk}{1,0,0.6,0.2}
\definecolor{myblue}{rgb}{0.66,0.78,1.00}
\definecolor{Violet}{cmyk}{0.79,0.88,0,0}
\definecolor{Lavender}{cmyk}{0,0.48,0,0}

\renewcommand{\epsilon}{\varepsilon}
\renewcommand{\phi}{\varphi}

\title{Orbits inside Fatou sets }

\vspace{5cm}
\author{   
	John Erik Forn\ae ss,\\
	\small Department of Mathematical Sciences, Norwegian University of Science and Technology,\\
	 \small Trondheim, 7034, Norway\\
	\and	Mi Hu 
	\thanks{The author is partially supported by Young Research Talents grant 300814 from the Research Council of Norway.} \\
	\small Department of Mathematics, University of Oslo, Oslo, 0371, Norway\\
	\small  Department of Mathematical, Physical and Computer Sciences, University of Parma, \\
	\small Parma, 43124, Italy\\
	\small	In memory of Professor Zhihua Chen.
	}

	\vspace{5cm}

\begin{document}

		\maketitle
		\begin{abstract}
			In this paper, we investigate the precise behavior of orbits inside attracting basins of rational functions on $\mathbb P^1$ and  entire functions $f$ in $\mathbb{C}$. 
			
				Let $R(z)$ be a rational function on $\mathbb P^1$,
			$\mathcal {A}(p)$ be the basin of attraction of an attracting fixed point $p$ of $R$, and $\Omega_i$ $ (i=1, 2, \cdots)$ be the connected components of $\mathcal{A}(p)$, and $\Omega_1$ contains $p.$ Let $p_0\in\Omega_1$ be close to $p.$
			If at least one $\Omega_i$ is not simply connected, then there exists a constant $C$ so that for any $z_0\in \Omega_i$, there
			is a point $q\in \cup_k R^{-k}(p_0), k\geq0$ so that the Kobayashi distance $d_{\Omega_i}(z_0, q)\leq C.$ If all $\Omega_i$ are simply connected, then the result is the same as for polynomials and is treated in an earlier paper.
			
			For  entire functions $f$, we generally can not have similar results as for  rational functions. However, if $f$ has finitely many critical points, then  similar results hold.

		\end{abstract}

		\section{Introduction}\label{sec1}
		
		A general goal in discrete dynamical systems is to qualitatively and quantitatively describe the possible dynamical behaviour under the iteration of maps satisfying certain conditions.

		Let $\mathbb P^1=\mathbb{C}\cup\{\infty\}, R: \mathbb P^1 \rightarrow \mathbb P^1$ be a nonconstant holomorphic map, and $R^{ n}:\mathbb P^1 \rightarrow \mathbb P^1$ be its $n$-fold iterate. In complex dynamics, two crucial disjoint invariant sets are associated with $R$, the {\sl Julia set} and the {\sl Fatou set} \cite{RefB, RefBM, RefCG, RefM}, which partition the sphere $\mathbb P^1$.
		The Fatou set of $R$ is defined as the largest open set where the family of iterates is locally normal. In other words, for any point $z\in \mathbb P^1$, there exists some neighborhood $U$ of $z$ so that the sequence of iterates of the map restricted to $U$ forms a normal family, so the iterates are well-behaved. 
		The complement of the Fatou set is called the Julia set. 
		The connected components of the Fatou set of $R$ are called {\sl Fatou components}. 
		A Fatou component $\Omega\subset \mathbb P^1$ of $R$ is {\sl invariant} if $R(\Omega)=\Omega$. 
		The sequence $\{z_n\}=\{z_1=R(z_0), z_2=R^2(z_0), \cdots\}$, where $z\in\mathbb P^1$, is referred to as the orbit of the point $z=z_0$.
		If there exists an integer $N$ such that $z_N=z_0$, we classify $z_0$ as a periodic point of $R$.
		In the special case where $N=1$, $z_0$ is then recognized as a fixed point of $R$.
		
		In the early 20th century, Fatou \cite{RefFatou} and Julia
		started to conduct a classification of all potential invariant
		Fatou components for rational functions on the Riemann sphere. In the 1930’s,  Cremer completed the classification of rotation domains using the Uniformization Theorem.  The existence of rotation domains was proved even later than	that, by Siegel and Herman. The classification of invariant Fatou components \cite{RefB, RefBM, RefCG, RefM} shows that if $f$ is a rational map of degree $d\geq 2$ and $\Omega=f(\Omega)$ is an invariant component of the Fatou set, then it is one of the following cases:
		\begin{enumerate}
			\item (attracting case) $\Omega$ contains a fixed point $p,$ and the orbit of every point in $\Omega$ converges to $p$. 
			\item (parabolic case) $\partial\Omega$ contains a fixed point $p,$ and the orbit of every point in $\Omega$ converges to $p$.  
			\item (rotation domain) $\Omega$ is conformally equivalent to the unit disk or an annulus, and the map is conjugate to an irrational rotation.
		\end{enumerate}
		
		The classification of Fatou components was completed in the 1980s when Sullivan proved that every Fatou component of a rational map is preperiodic, i.e., there are $n,m \in \mathbb{N}$ such that $R^{n+m}(\Omega)=R^m(\Omega)$. For a comprehensive understanding of additional details and results, see the paper \cite{RefS} for further information.

		However, there have been few detailed studies until now of the more precise behavior of orbits inside the Fatou set. 
		For example, let $\mathcal{A}(p):=\{z\in \mathbb{C}; f^n(z)\rightarrow p\}$ be the {\sl basin of attraction} of an attracting fixed point $p$. 
		An intriguing question arises when $z_0$ is in proximity to $\partial\mathcal{A}(p)$: what is the  behavior of the orbits ${z_n}$ starting from $z_0$ and approaching the vicinity of the attracting fixed point $p$?
		Specifically, if an orbit ${z_n}$ starting from $z_0$ eventually reaches the point $p$, we can gain insight into the behavior of the orbit of $z_0$ by examining the backward orbit of the fixed point $p$. However, there are cases where the orbit $\{z_n\}$ of $z_0$ approaches $p$ indefinitely without ever reaching it. In such situations, when considering the behavior of the orbit of $z_0$, it becomes uncertain whether we can rely on the backward orbit of the fixed point $p$ instead. Consequently, determining how to handle these situations becomes an interesting question. Additionally, one might wonder how many iterations are required to reach this uncertain state.
		One application arises from Newton's method in \cite{RefB}. 
		It is of practical interest to know how many times Newton's method must be iterated to get the desired approximation of the root. 
		
		In the paper \cite{RefH1}, Hu investigated the precise behavior of orbits inside attracting basins of polynomials in $\mathbb{C}$. Let $f$ be a holomorphic polynomial of degree $m\geq2$ in $\mathbb{C}$, $\mathcal {A}(p)$ be the basin of attraction of an attracting fixed point $p$ of $f$, and $\Omega_i$ $(i=1, 2, \cdots)$ be the connected components of $\mathcal{A}(p)$. Assume $\Omega_1$ contains $p$ and $\{f^{-1}(p)\}\cap \Omega_1\neq\{p\}$. Then there is a constant $C$ so that for every point $z_0$ inside any $\Omega_i$, there exists a point $q\in \cup_k f^{-k}(p)$ inside $\Omega_i$ such that $d_{\Omega_i}(z_0, q)\leq C$, where $d_{\Omega_i}$ is the Kobayashi distance on $\Omega_i.$ It shows that in an attracting basin of a complex polynomial in $\mathbb{C}$, the backward orbit of the attracting fixed point either is the point itself or accumulates
		at the boundary of all the components of the basin in such a way that all points of the basin lie within a uniformly bounded distance of the backward orbit, measured with respect to the Kobayashi metric. However, in paper \cite{RefH2}, Hu proved that this result is not valid for parabolic basins of polynomials in $\mathbb{C}$.
		Hu \cite{RefH3} also studied similar problems in higher dimensions and obtained a variety of interesting results. 
		
		It is natural to inquire about the precise behavior of orbits within the attracting basins of rational functions on $\mathbb P^1$. Exploring these questions forms the primary focus of this paper. In the second section, we give some basic definitions of the Kobayashi metric. Then, we discuss when a basin of infinity is not simply connected for polynomials in the third section. In the fourth section, we delve into the dynamics of rational functions within the attracting basins and establish our main result, namely Theorem A.

		\begin{thmA}
			Let $R(z)$ be a rational function on $\mathbb P^1$,
			$\mathcal {A}(p)$ be the basin of attraction of an attracting fixed point $p$ of $R$, and $\Omega_i$ $ (i=1, 2, \cdots)$ be the connected components of $\mathcal{A}(p)$. Assume $\Omega_1$ contains $p.$
			
			If all $\Omega_i$ are simply connected  and $\{R^{-1}(p)\}\cap \Omega_1\neq\{p\}$, then there is a constant $C$ so that for every point $z_0$ inside any $\Omega_i$, there exists a point $q\in \cup_k R^{-k}(p)$ inside $\Omega_i$ such that $d_{\Omega_i}(z_0, q)\leq C$, where $d_{\Omega_i}$ is the Kobayashi distance on $\Omega_i.$

			Suppose all $\Omega_i$ are simply connected  and $\{R^{-1}(p)\}\cap \Omega_1=\{p\}$; 	or at least one $\Omega_i$ is not simply connected. Let $p_0\in\Omega_1$ close to $p.$ Then there exists a constant $C$ so that for any $z_0\in \Omega_i$, there
			is a point $q\in \cup_k R^{-k}(p_0), k\geq0$ so that the Kobayashi distance
			$d_{\Omega_i}(z_0, q)\leq C.$
		\end{thmA}

		These theorems generalize the results in \cite{RefH1}
		which are for polynomials in $\mathbb C.$
		The key difference with rational functions is that the
		basins might not be simply connected.  This applies as well to the basin of infinity
		for a polynomial. It was observed in \cite{RefH1} that the theorem holds if the
		basin of attraction of $\infty$ of a polynomial is simply connected in $\mathbb P^1.$
		
		In the fifth section, we also discuss the corresponding result for entire functions in $\mathbb C.$
		There, we show that the theorem fails in general. Nevertheless, the theorem still holds
		after some mild conditions. (Finitely many critical points.)

		\section{The Kobayashi metric}\label{subsec2}
		First of all, let us introduce some definitions of the Kobayashi metric.
		
		\begin{defn}\label{def1}
			Let $\hat{\Omega}\subset\mathbb P^1$ be a domain. We choose a point $z\in \hat{\Omega}$ and a vector $\xi$ which is tangent to the plane at the point $z.$ Let $\triangle$ denote the unit disk in the complex plane.
			We define the {\em Kobayashi metric} \cite{RefK}
			$$
			F_{\hat{\Omega}}(z, \xi):=\inf\{\lambda>0 : \exists f: \triangle\stackrel{hol}{\longrightarrow} \hat{\Omega}, f(0)=z, \lambda f'(0)=\xi\}.
			$$

			Let $\gamma: [0, 1]\rightarrow \hat{\Omega}$ be a piecewise smooth curve.
			The {\em Kobayashi length} of $\gamma$ is defined to be 
			$$ L_{\hat{\Omega}} (\gamma)=\int_{\gamma} F_{\hat{\Omega}}(z, \xi) \lvert dz\rvert=\int_{0}^{1}F_{\hat{\Omega}}\big(\gamma(t), \gamma'(t)\big)\lvert \gamma'(t)\rvert dt.$$

			For any two points $z_1$ and $z_2$ in $\hat{\Omega}$, the {\em Kobayashi distance} between $z_1$ and $z_2$ is defined to be 
			$$d_{\hat{\Omega}}(z_1, z_2)=\inf\{L_{\hat{\Omega}} (\gamma): \gamma ~ \text{is a piecewise smooth curve connecting} ~z_1~ \text{and} ~z_2 \}.$$

			Note that $d_{\hat{\Omega}}(z_1, z_2)$ is defined where $z_1, z_2$ are in the same connected component of $\hat{\Omega}.$

			Let $d_E(z_1, z_2)$ denote the Euclidean metric distance for any two points $z_1, z_2\in\triangle.$ 
			
		\end{defn}

		We know that if $\hat{\Omega}=\triangle$, then the Kobayashi metric is the same as the Poincar\'{e} metric (see page 9 in \cite{RefF} and sections $0$ and $3$ in \cite{RefK}). And
		$$
		F_{\triangle}(z, \xi)=\frac{\lvert \xi\rvert}{1-\lvert z\rvert^2}.
		$$
		Note that it is common to call $F_{\hat{Q}}$ a "Poincar\'{e} metric" for any $\hat{Q}\subset \mathbb {C},$ but for generalization to higher dimensions, and also for the basin of $\infty$, we call it a "Kobayashi metric" even in dimension one to be more consistent.

		\begin{prop}[The distance decreasing property of the Kobayashi Metric \cite{RefK}]\label{pro1}
			Suppose $\Omega_1, \Omega_2$ are domains in $\mathbb P^1$, $z, \omega\in \Omega_1, \xi\in\mathbb C,$ and $f:\Omega_1\rightarrow\Omega_2$ is holomorphic. Then 
			$$F_{\Omega_2}(f(z
			), f'(z)\xi)\leq F_{\Omega_1}(z, \xi), ~~~d_{\Omega_2}(f(z), f(\omega))\leq d_{\Omega_1}(z,\omega).$$
		\end{prop}

		\begin{cor}\label{cor1}
			Suppose $\Omega_1\subseteq\Omega_2\subseteq\mathbb C.$ Then for any $z, \omega\in \Omega_1$ and $\xi\in\mathbb C,$ we have 
			$$F_{\Omega_2}(z, \xi)\leq F_{\Omega_1}(z, \xi), ~~~~d_{\Omega_2}(z, \omega)\leq d_{\Omega_1}(z,\omega).$$
		\end{cor}

		\section{Dynamics of polynomials inside attracting basins in $\mathbb{P}^1$}\label{sec3}
		
		In the paper \cite{RefH1}, for a polynomial $f$ in $\mathbb{C}$ or if the basin of attraction $\mathcal{A}(\infty)$  of $f$ is simply connected, Hu proved the following theorem and corollary: 
		
		\begin{thm}\label{the3}
			Suppose $f(z)$ is a polynomial of degree $N\geq 2$ on $\mathbb{C}$, $p$ is an attracting fixed point of $f(z),$ $\Omega_1$ is the immediate basin of attraction of $p$, $\{f^{-1}(p)\}\cap \Omega_1\neq\{p\}$,
			$\mathcal{A}(p)$ is the basin of attraction of $p$, $\Omega_i $ $(i=1, 2, \cdots)$ are the connected components of $\mathcal{A}(p)$. Then there is a constant $\tilde{C}$ so that for every point $z_0$ inside any $\Omega_i$, there exists a point $q\in \cup_k f^{-k}(p)$ inside $\Omega_i$ such that $d_{\Omega_i}(z_0, q)\leq \tilde{C}$, where $d_{\Omega_i}$ is the Kobayashi distance on $\Omega_i.$  
		\end{thm} 
		\begin{cor}\label{cor2}
			Let $f$ be a polynomial of degree $m\geq2$, considered as a map on $\hat{\mathbb{C}}.$ Suppose that $\Omega(\infty) $ is the basin of $\infty$ and is simply connected. If $p\in\Omega\setminus\{\infty\},$ then there exists a constant $C_0>0$ such that for every point $z_0\in\Omega$, there exists $q\in \cup_k g^{-k}(p), k\geq0$ satisfying $d_\Omega(z_0,q)\leq C_0,$ where $d_\Omega$ is the Kobayashi distance on $\Omega.$
		\end{cor}
		
		It is very natural to ask, when the basin of infinity is not simply connected, whether we can still obtain analogous results. 
		To answer it, the following theorem comes into play:

		\begin{thm}\label{them1}
			Let $f$ be a polynomial in $\mathbb{P}^1$ and $\mathcal{A}(\infty)$ be the basin of attraction of infinity. Let $p\in\mathcal{A}(\infty)$ close to $\infty$, $p \neq \infty.$
			Then there exists a constant $C$ so that for any $z_0\in \mathcal{A}(\infty)$, there
			is a point $q\in \cup_k f^{-k}(p), k\geq0$ so that the Kobayashi distance
			$d_{\mathcal{A}(\infty)}(z_0, q)\leq C.$
		\end{thm}
		
		\begin{proof}
			Let $U_0$ be a small neighborhood of $\infty$ containing $p.$
			We let in fact $U_0$ be a disc in B\"{o}ttcher coordinates centered at $\infty.$
			Let $U_n=f^{-n}(U_0).$
			There exists an $n_0$ so that $U_{n_0}$ contains all critical points inside $\mathcal{A}(\infty)$.
			Set $W_{n+1}=U_{n+1}\setminus \overline{U}_n.$
			
			Note that we need to be careful about the topology. The sets $W_{n+1} $ might have many connected components. We need to make sure that the orbit
			of $z_0$ lands in the same connected component as one of the $f^{-k}(p).$ Hence, we need the following lemma.
			\begin{lem}\label{lem1}
				The sets $W_{n}$ have finitely many connected components $W_{n,i}.$ There is at least one preimage of $p$ in each connected component of each $W_n.$
				
			\end{lem}
			\begin{proof}
				Let $U_0$ be a small neighborhood of $\infty$ containing $p$ and $U_n=f^{-n}(U_0).$ Using B\"{o}ttcher coordinates we know that $f$ is equivalent to 
				$z^N$ ( $N$ is actually the degree of $f$) near infinity. So we can choose $U_0$ to be a small disc in the B\"{o}ttcher coordinates.
				We then get that $U_1= f^{-1}(U_0)$ is also such a disc with larger radius and
				$f$ is an unbranched cover from $U_1$ to $U_0$ of order $N$.
				Inductively $U_{n}=f^{-n}(U_0)$ is a branched or unbranched cover of order $N$ of $U_{n-1}$. 
				Then, all the $U_n$ will be connected. 
				
				Next, we need to know whether $W_n$ is connected. This is true when $n$ is small, say $n\leq n'$.
				But then $W_{n'+1}$ will have at most $N$ different components, 
				each of which is a cover of $W_{n'}$.
				Similarly, $W_{n}$ will have finitely many different components for all larger $n>n'+1.$

				The map $f:W_{n+1}\rightarrow W_n$ is a branched or unbranched cover of order $N.$ 
				Each $W_n$ is a finite union of connected components $W_{n, i}.$
				For each $W_{n, i}$ there are finitely many connected components of
				$W_{n+1}$ which are branched or unbranched covers of $W_{n, i}.$
				
				Suppose that for each component $W_{n, i}$ of $W_n,$ 
				there is a preimage of the form $f^{-m}(p)\in W_{n, i}.$ 
				Pick any connected component $W_{n+1, j}$ of $W_{n+1}.$
				Then this is a branched cover of some $W_{n, i}$. 
				Hence there exists a preimage $f^{-m-1}(p)$ in this component $W_{n+1, j}$. By induction, there is at least one preimage of $p$ in each connected component of each $W_n.$ Using continuity of the Kobayashi distance, we can also also perturb $z_0$
				a little so that its orbit never intersects any $\partial U_n.$

			\end{proof}
			This lemma implies that when $f^n(z_0)$ is in some connected component $W_{m, i}$ of $W_{m},$ there always exists a preimage $q\in  \cup_k f^{-k}(p), k\geq0 $ in the same component $W_{m, i}.$
			Since we need to calculate the Kobayashi distance between $q$ and $f^n(z_0)$ on the same component $W_{m, i}$, we need to avoid any preimages $q$ landing on the boundary of any components of $W_{m, i}$. Thus, we add the extra assumption that there is no orbit of a critical point containing a point on the boundary of $U_0,$ and there is no preimage of $p$ on the boundary of $U_0.$ Then the boundary of each $U_n$ consists of smooth curves.
			
			We can then also find a curve $\gamma_n$ in the same component
			which connects $q$ to $f^n(z_0).$
			By induction we can lift $\gamma_n$ to $\gamma_k$ connecting $f^k(z_0)$
			to some preimage $f^{-\ell}(p).$ This whole curve will then be inside
			the same component of $W_{n-k, i.}$
			
			Let $V_n=\mathcal{A}(\infty) \setminus \overline{U}_n.$	Then the map $f: V_{N_0+1}\rightarrow V_{N_0}, N_0> n_0$ is an unbranched cover. 
			Thus, for any two points $z,w$ in $V_{N_0},$ 
			we can find preimages $z'=f^{-1}(z), w'=f^{-1}(w)$ in $V_{N_0+1}$
			so that 
			$$d_{V_{N_0+1}}(z', w')= d_{V_{N_0}}(z, w).$$

			By Corollary \ref{cor1}, we know
			
			$$d_{V_{N_0}}(z', w')\leq d_{V_{N_0+1}}(z', w').$$
			
			Hence, 
			\begin{equation}\label{eq1}
				d_{V_{N_0}}(z', w')\leq d_{V_{N_0}}(z, w).
			\end{equation}
			
			Next, we need to consider $z_0$ inside $\mathcal{A}(\infty)$ or $z_0$ is very close to the boundary of $\mathcal{A}(\infty).$

			If $ z_0$ is far away from $\infty$ and close to $\partial\mathcal{A}(\infty)$, we know that $f^{N_1}(z_0)\in W_{1}$ for some integer $N_1.$ 
			We can then choose some $q':=f^{-M_1}(p)\subseteq W_{1}$, $M_1$ is a positive integer. 
			
			Then we let $\gamma'_1\in W_1$ be a curve connecting $f^{N_1}(z_0)$ and $q'.$ We lift $\gamma'_1$ to $\gamma'_2$, where $\gamma'_2$ connects $f^{N_1-1}(z_0)$ and some preimage $f^{-1}(q').$ This whole curve $\gamma'_2$ will be inside the same component of $W_{2}, i.$ Then by induction,  $\gamma'_{N_1}$ connects $z_0$ and some preimage $f^{-N_1}(q')$ will be inside some component of $W_{N_1, j}$ as well. 
			
			Suppose 
			$\gamma'_{\tilde{N}}\in W_{\tilde{N}}$ 
			connecting
			$f^{N_1+1-\tilde{N}}(z_0)$ and some preimage $f^{-\tilde{N}+1}(q')$ is the first curve of $\gamma'_1, \cdots, \gamma'_{N_1}$ landed in $V_{n_0+1}$ for some integer $1<\tilde{N}< N_1$. 
			Since $W_{\tilde{N}}\subset\subset V_{n_0+1}\subset\mathcal{A}(\infty),$ there is a constant $C_0$ so that $d_{W_{\tilde{N}}}\big(f^{N_1+1-\tilde{N}}(z_0), f^{-\tilde{N}+1}(q')\big)\leq C_0.$
			
			Then by  Corollary \ref{cor1}, we have 
			\begin{equation}\label{eq3}
				d_{\mathcal{A}(\infty)}(z_0, f^{-N_1}(q'))\leq d_{V_{n_0+1}}\big(z_0, f^{-N_1}(q')\big), 
			\end{equation}
			and
			\begin{equation}\label{eq4}
				d_{V_{n_0+1}}\big(f^{N_1+1-\tilde{N}}(z_0), f^{-\tilde{N}+1}(q')\big) \leq d_{W_{\tilde{N}}}\big(f^{N_1+1-\tilde{N}}(z_0), f^{-\tilde{N}+1}(q')\big).
			\end{equation}
			By the inequalities (\ref{eq1}), we have
			\begin{equation}\label{eq5}
				d_{V_{n_0+1}}\big(z_0, f^{-N_1}(q')\big)\leq 
				d_{V_{n_0+1}}\big(f(z_0), f^{-N_1+1}(q')\big)\leq
				\cdots\leq
				d_{V_{n_0+1}}\big(f^{N_1+1-\tilde{N}}(z_0), f^{-\tilde{N}+1}(q')\big) .
			\end{equation}
			Hence,  we have
			$$d_{\mathcal{A}(\infty)}(z_0, f^{-N_1}(q'))\leq C_0.$$
			
			If $z_0$ inside any $U_n\subset\mathcal{A}(\infty), n\leq n_0.$ We can choose some $q=p$. In this case, since $U_n\subset\subset\mathcal{A}(\infty),$
			there is a constant $C'_0$ so that $d_{\mathcal{A}(\infty)}(z_0, q)\leq C'_0.$
			
			Therefore, let $C=\max \{C_0, C'_0\},$ there always exist a point $q \in  \cup_k f^{-k}(p), k\geq0 $ such that 
			$$d_{\mathcal{A}(\infty)}(z_0, q)\leq C$$ for any $z_0\in\mathcal{A}(\infty).$

		\end{proof}

		\section{Dynamics of rational functions inside attracting basins in $\mathbb{P}^1$}\label{sec4}
		
		Let $R: \mathbb{P}^1\rightarrow\mathbb{P}^1$ be a rational function and
		$\Omega$ be an invariant Fatou component. Section \ref{sec1} provides a classification that offers a comprehensive understanding of the long-term behavior of all orbits within $\Omega$. However, the precise movement of iterates inside $\Omega$ remains unknown. This section aims to provide a more detailed description of how the iterates of a point evolve within an invariant Fatou component and the whole basin of attraction. Subsequently, we present our main results, namely Theorem A.

		\begin{thmA}\label{thm1}
			Let $R(z)$ be a rational function on $\mathbb P^1$,
			$\mathcal {A}(p)$ be the basin of attraction of an attracting fixed point $p$ of $R$, and $\Omega_i$ $ (i=1, 2, \cdots)$ be the connected components of $\mathcal{A}(p)$. Assume $\Omega_1$ contains $p.$
			
			If all $\Omega_i$ are simply connected  and $\{R^{-1}(p)\}\cap \Omega_1\neq\{p\}$, then there is a constant $C$ so that for every point $z_0$ inside any $\Omega_i$, there exists a point $q\in \cup_k R^{-k}(p)$ inside $\Omega_i$ such that $d_{\Omega_i}(z_0, q)\leq C$, where $d_{\Omega_i}$ is the Kobayashi distance on $\Omega_i.$

			Suppose all $\Omega_i$ are simply connected  and $\{R^{-1}(p)\}\cap \Omega_1=\{p\}$; 	or at least one $\Omega_i$ is not simply connected. Let $p_0\in\Omega_1$ close to $p.$ Then there exists a constant $C$ so that for any $z_0\in \Omega_i$, there
			is a point $q\in \cup_k R^{-k}(p_0), k\geq0$ so that the Kobayashi distance
			$d_{\Omega_i}(z_0, q)\leq C.$
		\end{thmA}
		\begin{proof}
			If all $\Omega_i$ are simply connected  and $\{R^{-1}(p)\}\cap \Omega_1\neq\{p\}$, then the theorem holds by Theorem B in the paper \cite{RefH1}. 
			
			Suppose all $\Omega_i$ are simply connected and $\{R^{-1}(p)\}\cap \Omega_1=\{p\}$, then we choose $q$ to be a preimage of $p_0.$ Then the theorem holds by Theorem B in the paper \cite{RefH1}.
			
			If at least one $\Omega_i$ is not simply connected, there are two possible cases.
			
			(1) $\mathcal{A}(p)=\Omega_1$. Then, in fact, the proof is the same as for Theorem \ref{them1} in section 3.
			
			(2) $\mathcal{A}(p)$ has at least two connected components, then $\mathcal{A}(p)$ has infinity many components (see the proof of Theorem B in the paper \cite{RefH1}). If $z_0$ is inside $\Omega_1,$ then it is the same as the case (1) above. If $z_0\in \Omega_2$ and $R(\Omega_2)=\Omega_1.$ We know $R(z_0)\in \Omega_1.$ There exists a constant $C$ so that for every $R(z_0)\in\Omega_1$, there exists a point $q\in \cup_k R^{-k}(p_0)\in\Omega_1$ such that $d_{\Omega_1}(R(z_0), q)\leq C.$ 
			Then we need to consider if we can find a preimage of $q$ inside $\Omega_2$ such that the Kobayashi distance between $R^{-1}(q)$ and $z_0$ can still be uniformly bounded.

			We choose a compact set $S_1\in\Omega_1$ such that all critical points inside $\Omega_1$ belongs to $S_1$ and no critical points on the boundary of $S_1.$ Note that if there is no critical points inside $\Omega_1,$ then we just choose any compact set inside $\Omega_1.$ Then we choose a compact set $S_2\in\Omega_2$ such that $R^{-1}(S_1)\subseteq S_2.$ 
			After that, we choose a compact set $S_3$ such that $S_3$ contains $S_2$ and all critical points which are inside $\Omega_2,$  and there are no critical points on the boundary of $S_3$. Note that  if $\Omega_2$ has no critical points, then we just need to choose a compact set $S_3$ such that $S_2\subset S_3.$ Let $D:=\Omega_2\setminus S_3.$ And we know $R(z_0)\in\Omega_1$ and there is a point $q\in\Omega_1$ so that $d_{\Omega_1}(R(z_0), q)\leqslant C.$ 
			Then we can choose a preimage of $q$ inside $\Omega_2.$

			And there are four cases for distributing $z_0$ and the preimage of $p_0$ we want in $\Omega_2.$
			\begin{enumerate}
				\item 
				If $z_0, R^{-1}(q)\in S_3,$ then there is a constant $C'$ so that  $d_{S_3}(z_0, R^{-1}(q))\leq C'$ since $S_3\subset\subset \Omega_2.$ Thus, by Corollary \ref{cor1}, $d_{\Omega_2}(z_0, R^{-1}(q))\leq d_{S_3}(z_0, R^{-1}(q))\leq C'.$
				
				\item If $z_0\in S_3, q\in D,$ then we can choose a compact set $S_4\in\Omega_2$ including $z_0, R^{-1}(q)$ and $S_3\subset S_4.$ Then there is a constant $C''$ such that $d_{S_4}(z_0, R^{-1}(q))\leq C''$ since $S_4\subset\subset \Omega_2.$ Hence, by Corollary \ref{cor1} again, $d_{\Omega_2}(z_0, R^{-1}(q))\leq d_{S_4}(z_0, R^{-1}(q))\leq C''.$

				\item If $z_0\in D, q\in S_3,$ then this case is the same as case 2. 
				\item If $z_0, q\in D,$ 
				then $R(z): D\rightarrow \Omega_1$ is an unbranched covering.  
				We choose a compact subset $S_5\in D$ so that 
				and $z_0, R^{-1}(q)$ inside $S_5,$ but $z_0, R^{-1}(q)$ are neither on the boundary of $S_5$ nor very close to $\partial S_5.$ Otherwise, we choose a litter big compact set to be $S_5.$
				Then by Corollary \ref{cor1}, we know 
				$$d_{\Omega_2}(z_0, R^{-1}(q))\leq d_{S_5}(z_0, R^{-1}(q)).$$
				Next,  we need to show there is a constant $C'''$ such that $d_{S_5}(z_0, R^{-1}(q))\leq C'''$. 
				This is true by Theorem 3.4 in Wold’s
				paper \cite{RefW}. The proof is essentially the same as the proof of Case 4 in the proof of Theorem B in Hu's paper \cite{RefH1}.

			\end{enumerate}

			We choose $C=\max\{ C', C'', C'''\}, $ for every point $z_0$ inside any $\Omega_2$, there exists a point $q\in \cup_k R^{-k}(p_0)$ inside $\Omega_2$ such that $d_{\Omega_2}(z_0, q)\leq C$, where $d_{\Omega_2}$ is the Kobayashi distance on $\Omega_2.$ 
			
			If $z_0$ inside some other component $\Omega_i$ so that $R^{j}(\Omega_i)=\Omega_1,$ we repeat the procedure above in case (2), we still can prove this theorem holds.
			
			Therefore, there is a constant $C$ for every point $z_0$ inside any $\Omega_i$, there exists a point $q\in \cup_k R^{-k}(p)$ inside $\Omega_i$ such that $d_{\Omega_i}(z_0, q)\leq C$, where $d_{\Omega_i}$ is the Kobayashi distance on $\Omega_i.$

		\end{proof}

		\section{Dynamics of entire functions inside attracting basins in $\mathbb{C}$}\label{sec5}
		
		In this section, we also discuss the corresponding result for entire functions in $\mathbb C.$
		There, we show that the theorem fails in general. Nevertheless, the theorem still holds
		after some mild conditions. (Finitely many critical points.)
		
		We construct an entire function $f(z)$ with an attracting fixed point 
		so that Theorem A fails.
		
		We want $f(z)$ to have the following properties:\\
		(i) $f(0)=0, |f'(0)|<1;$\\
		(ii) There exists an increasing  sequence of strictly positive numbers
		$A_n\rightarrow \infty$,
		so that $f(\{A_{2n}<|z|<A_{2n+1}\})\subset \subset\{A_{2n+2}<|z|<A_{2n+3}\}$;\\
		(iii) There exists a sequence $z_n, A_{2n+1}<|z_n|<A_{2n+2}$ so that
		$f(z_n)=0 $ and $f$ vanishes to at least order $n$ at $z_n.$
		
		We construct $f$ inductively. Suppose that $f$ is an entire function
		and that $A_{2n}<A_{2n+1}$ Choose a new $f$, we do this in two steps.
		First pick a point $z_n$ with $|z_n|> A_{2n+1}.$ We first let the new $f$ be given by
		$f_1= f (z-z_n)^n e^{h_n}$ where $h_n$ is an entire function which approximates  $-\ln (z-z_n)^n$ very well on $|z|<A_{2n+1}$.
		Then $f_1$ approximates $f$ as well as we want on $|z|<A_{2n+1}$
		and also $f_1$ vanishes to order $n$ at $z_n.$
		Next, we let 
		$f_2= f_1+ (\frac{2z}{A_{2n}})^N$. Then for $N$ large enough, $f_2$ and $f$ are as close as we want on $|z|<A_{2n}/4$. Moreover,
		the image of $\{A_{2n}<|z|<A_{2n+1}\}\subset \{A_{2n+2}<|z|<A_{2n+3}\}$
		where $A_{2n+2 }$ can be chosen as large as we want and in particular bigger than $|z_n|.$

		Let $\Omega$ be the basin of attraction of $0$ and let $U_0$ be the immediate basin. Then $U_0$ is simply connected. 
		Let $U_n$ be the Fatou component containing $z_n.$
		Then $U_n\subset \{A_{2n+1}<|z|<A_{2n+2}\}$. By the maximum principle, it must be simply connected. The map $f:U_n\rightarrow U_0$ is conjugate to 
		a map $F$ on the unit disc with the property that $F(z)=z^n g$.
		Consider the point $1/2$ in $U_0$. The primages of this point will be arbitrarily close to the boundary of $U_n$. Hence, the Kobayashi distance will go to infinity. So there is no uniform constant $C.$
		
		On the other hand, we have the following positive results:
		
		\begin{thm}
			Suppose that $f(z)$ is an entire function with $0$ as an attracting fixed point with nonzero derivative at the origin. Suppose that
			$f$ has finitely many critical points. Suppose also that all Fatou components are bounded open sets. Then Theorem A holds.
		\end{thm}
		
		\begin{proof}
			Since all Fatou components are bounded, they are also simply connected.
			On the immediate basin of attraction, the map behaves similarly with
			a polynomial on the immediate basin of attraction. If we consider the 
			map $f:\Omega_1\rightarrow \Omega_2$ on preimages of the immediate basin, $f$ will be biholomorphic if there is no critical point in $\Omega_1$.
			Furthermore, only for finitely many choices of $\Omega_1$ will there be a critical point inside. So the situation is the same as for a polynomial.
		\end{proof}


\begin{thebibliography}{}
			
			
			
			
			
			\bibitem{RefB} Beardon, A.F.: Iteration of Rational Functions. Springer-Verlag, New York (1991)
			
			\bibitem{RefBM}	Bergweiler, W., Morosawa, S.: Semihyperbolic entire functions, Nonlinearity, 15, 1673 - 1684 (2002)
			
			\bibitem{RefCG} Carleson, L., Gamelin, T. W.: Complex dynamics. Springer-Verlag, New York (1993)   
			
			\bibitem{RefF} Fornæss, J. E.: Dynamics in Several Complex Variables, CBMS Regional Conference Series in Mathematics (1996)
			
			\bibitem{RefFatou} 	Fatou, P.: \'{E}quations fonctionnelles, Bull. Soc. Math. France, 47, 161-271 (1919); 48, 33-94 (1920)
			

			
			
			
			
			\bibitem{RefH1} Hu, M., Interior Dynamics of Fatou Sets,  Annali di Matematica, 2023.  https://doi.org/10.1007/s10231-023-01344-9
			
			\bibitem{RefH2} Hu, M., Dynamics inside Parabolic Basins, arXiv preprint, arXiv:2208.03756, 2022.
			
			\bibitem{RefH3} Hu, M., Dynamics inside Fatou Sets in Higher Dimensions, Annali della Scuola Normale Superiore di Pisa, Classe di Scienze, 2023. 
			https://doi.org/10.2422/2036-2145.202301-004                
			
			\bibitem{RefK} Krantz, S.: The Caratheodory and Kobayashi Metrics and Applications in Complex Analysis, The American Mathematical Monthly, 115, 304-329 (2008) 
			
			\bibitem{RefM} Milnor, J.: Dynamics in One Complex Variable. Princeton University Press, Princeton (2006) 
			
			\bibitem{RefS} Sullivan, D.: Quasiconformal Homeomorphisms and Dynamics I. Solution of the Fatou-Julia Problem on Wandering Domains, Annals of Mathematics, Second Series, Vol. 122, No. 2, 401-418 (1985)
			
			\bibitem{RefW} Wold, E. F.: Asymptotics of Invariant Metrics in the normal direction and a new characterisation of the unit disk, Math. Z. 288, 875–887 (2018)
			
			
			
		\end{thebibliography}
	\end{document}